\documentclass[11pt]{article}
\usepackage{amsthm}
\usepackage{graphicx, amssymb, amstext, amsmath, pdfsync, tikz, amscd}
\usepackage{tikz}
\usepackage{tikz}
\usetikzlibrary{calc}
\usetikzlibrary{matrix,arrows}
\usepackage{comment}
\usepackage[english]{babel}
\usepackage[tmargin=1in,bmargin=1in,lmargin=.5in,rmargin=.5in]{geometry}

\usepackage{hyperref}
\hypersetup{
    colorlinks=true,
    linkcolor=blue,
    filecolor=magenta,      
    urlcolor=cyan,
}

\newtheorem{definition}{Definition}[section]
\newtheorem{theorem}[definition]{Theorem}
\newtheorem{lemma}[definition]{Lemma}
\newtheorem{proposition}[definition]{Proposition}

\theoremstyle{remark}

\newtheorem{remark}[definition]{Remark}
\begin{document}
\title{A generalization of Grothendieck's Extension Panach\'ees}
\author{Rakesh R. Pawar}
\maketitle
\abstract We formulate a generalization of the extension problem for exact sequences which was considered in~\cite{SGA 7} and give a necessary and sufficient criterion for the solution to exist. We also remark on the criterion under which such a solution is unique, if it exists. 
\begin{section}{Introduction}

Let $\mathcal{A}$ be an abelian category. All the objects and arrows below are in $\mathcal{A}$.
Given a diagram with rows and columns as short exact sequences 

\begin{equation}
\label{1}
\begin{tikzpicture}
[back line/.style={densely dotted},cross line/.style={loosely dotted}]
\matrix (m) [matrix of math nodes, row sep=2 em,column sep=2.em, text height=1.5ex, text depth=0.50ex]
 {& 0 & 0 & 0 &\\
 0 & P & E & R & 0\\
 0 & H && F & 0\\
 0 & S & G & Q & 0\\
  & 0 & 0 & 0 &\\};
\path[->,font=\scriptsize]
  (m-1-2) edge node[auto] {} (m-2-2)
  (m-1-3) edge node [auto] {} (m-2-3)
(m-1-4) edge node [auto] {} (m-2-4)
 (m-2-1) edge node [auto] {} (m-2-2)
  (m-2-2) edge node [auto] {} (m-2-3)
    edge node [auto] {} (m-3-2)
  (m-2-3) edge node [auto] {} (m-2-4)
   
   (m-2-4) edge node [auto] {} (m-2-5)
   edge node [auto] {} (m-3-4)
   (m-3-1) edge node [auto] {} (m-3-2)
  (m-3-2) edge node [auto]{} (m-4-2)  
  (m-3-4) edge node [auto] {} (m-3-5) 
              edge node [auto] {} (m-4-4)   
  (m-4-1) edge node [auto] {} (m-4-2) 
     (m-4-2) edge node [auto]{}(m-4-3)
               edge node [auto]{}(m-5-2)
  (m-4-3) edge node [auto] {} (m-4-4)
               edge node [auto] {} (m-5-3)       
  (m-4-4) edge node [auto] {} (m-4-5)
              edge node [auto] {} (m-5-4);
\end{tikzpicture}
\end{equation}¥
we would like to complete the diagram to one with exact rows and columns :
\begin{equation}
\label{2}
\begin{tikzpicture}
[back line/.style={densely dotted},cross line/.style={loosely dotted}]
\matrix (m) [matrix of math nodes, row sep=2. em,column sep=2.em, text height=1.5ex, text depth=0.50ex]
 {& 0 & 0 & 0 &\\
 0 & P & E & R & 0\\
 0 & H & X & F & 0\\
 0 & S & G & Q & 0\\
  & 0 & 0 & 0 &\\};
 \path[->, densely dashed]  (m-2-3) edge node [auto] {} (m-3-3)
 (m-3-2) edge node [auto]{} (m-3-3)
  (m-3-3) edge node [auto] {} (m-3-4)
   edge node [auto] {} (m-4-3)   ;
\path[->,font=\scriptsize]
  (m-1-2) edge node[auto] {} (m-2-2)
  (m-1-3) edge node [auto] {} (m-2-3)
(m-1-4) edge node [auto] {} (m-2-4)
 (m-2-1) edge node [auto] {} (m-2-2)
  (m-2-2) edge node [auto] {} (m-2-3)
    edge node [auto] {} (m-3-2)
  (m-2-3) edge node [auto] {} (m-2-4)
   (m-2-4) edge node [auto] {} (m-2-5)
           edge node [auto] {} (m-3-4)
   (m-3-1) edge node [auto] {} (m-3-2)
  (m-3-2) edge node [auto]{} (m-4-2) 
  (m-3-4) edge node [auto] {} (m-3-5) 
              edge node [auto] {} (m-4-4)   
  (m-4-1) edge node [auto] {} (m-4-2) 
     (m-4-2) edge node [auto]{}(m-4-3)
               edge node [auto]{}(m-5-2)
  (m-4-3) edge node [auto] {} (m-4-4)
               edge node [auto] {} (m-5-3)       
  (m-4-4) edge node [auto] {} (m-4-5)
              edge node [auto] {} (m-5-4);
\end{tikzpicture}
\end{equation}
We will give a necessary and sufficient condition to complete this diagram. 
Note that when $S=0$ we get the \textbf{Extensions Panach\'ees} discussed by A. Grothendieck in~\cite{SGA 7} section 9.3. So this is a proper generalization of that setup. 

Let us start with the diagram \ref{1}. Since the rows and columns of that diagram are exact sequences, the exact sequence $0\to P\to E\to R\to 0$ defines an element $[E]\in Ext^1(R, P)$. Similarly we get $[H] \in Ext^1(S, P), [F] \in Ext^1(Q, R), [G] \in Ext^1(Q, S)$. 
Now $[E]\cup [F]$ is defined to be the class $[0\to P\to E\to F\to Q\to 0]$ as an element in $Ext^2(Q, P)$. Similarly have $[H]\cup [G]\in Ext^2(Q, P)$. The the addition of these elements in $Ext^2(Q, P)$ by is given by Baer sum. 

We have the following theorem:
\begin{theorem}
\label{1.1}
The following are equivalent:
\begin{enumerate}
\item[(i)] The diagram \ref{1} can be completed to diagram \ref{2}.
\item[(ii)] The Baer sum of $[E]\cup [F]$ and $[H]\cup [G]$ is 0.
¥
\end{enumerate}¥
\end{theorem}
The notations will be explained in the next section. The proof will follow from Propositions \ref{2.1} and \ref{2.5}.
We will also observe that under certain conditions, see Remark \ref{sol}, if the diagram \ref{1} can be completed to diagram \ref{2} then it can be done in a unique manner. 
\paragraph {Acknowledgement} 
The question considered in this paper was asked by Prof. B. Kahn and was communicated to the author by Prof. V. Srinivas. The author would like to thank Prof. V. Srinivas for subsequent discussions and guidance that led to this paper. The author would also like to thank Prof. B. Kahn for email communication that led to the formulation of the condition (ii) in the Theorem~\ref{1.1}. The author is supported by SPM fellowship (File no: SPM-07/858(0139)/2012) funded by CSIR. 
\end{section}
\begin{section}{Proof of Theorem 1.1}
\paragraph{Notations}

The strategy to prove this result is an auxiliary construction of an element $\delta_Y([W])$ in $Ext^2(Q, P)$, which will give an obstruction to complete the diagram \ref{1} to diagram \ref{2} [Prop. \ref{2.1}].  \\ 
Consider the exact sequence $0\to R\to F\to Q\to 0$. Pull back this exact sequence via the map $G\to Q$. So we get $Y$ such that it fits into this diagram with exact rows. 
\begin{equation}
\begin{tikzpicture}
[back line/.style={densely dotted},cross line/.style={loosely dotted}]
\matrix (m) [matrix of math nodes, row sep=2 em,column sep=2.em, text height=1.5ex, text depth=0.50ex]
{0 & R & Y & G & 0\\
 0 & R & F & Q & 0\\
  };

\path[->,font=\scriptsize]
 (m-1-1) edge node [auto] {} (m-1-2)
  (m-1-2) edge node [auto] {id} (m-2-2)
   (m-1-2) edge node [auto] {} (m-1-3)
  (m-1-3) edge node [auto] {} (m-1-4)
              edge node [auto] {} (m-2-3)
   (m-1-4) edge node [auto] {} (m-1-5)
           edge node [auto] {} (m-2-4)
   (m-2-1) edge node [auto] {} (m-2-2)
  (m-2-2) edge node [auto]{} (m-2-3) 
   (m-2-3) edge node [auto] {} (m-2-4)
  (m-2-4) edge node [auto] {} (m-2-5) ;
\end{tikzpicture}
\label{3}
\end{equation}
First we observe by Snake lemma applied to diagram \ref{3}, that $\exists$ $S\to Y$ such that it fits into the following diagram with exact rows and columns:
\begin{equation}
\label{4}
\begin{tikzpicture}
[back line/.style={densely dotted},cross line/.style={loosely dotted}]
\matrix (m) [matrix of math nodes, row sep=2 em,column sep=2.em, text height=1.5ex, text depth=0.50ex]
 {&  & 0 & 0 &\\
  &  & R & R & 0\\
 0 & S & Y & F & 0\\
 0 & S & G & Q & 0\\
  & 0 & 0 & 0 & \\};
\path[->,font=\scriptsize]
  (m-1-3) edge node [auto] {} (m-2-3)
(m-1-4) edge node [auto] {} (m-2-4)
  (m-2-3) edge node [auto] {id} (m-2-4)
              edge node [auto] {} (m-3-3)
   (m-2-4) edge node [auto] {} (m-2-5)
           edge node [auto] {} (m-3-4)
   (m-3-1) edge node [auto] {} (m-3-2)
  (m-3-2) edge node [auto]{id} (m-4-2) 
   edge node [auto]{} (m-3-3) 
   (m-3-3) edge node [auto] {} (m-3-4)
  edge node [auto] {} (m-4-3)  
  (m-3-4) edge node [auto] {} (m-3-5) 
              edge node [auto] {} (m-4-4)   
  (m-4-1) edge node [auto] {} (m-4-2) 
     (m-4-2) edge node [auto]{}(m-4-3)
               edge node [auto]{}(m-5-2)
  (m-4-3) edge node [auto] {} (m-4-4)
               edge node [auto] {} (m-5-3)       
  (m-4-4) edge node [auto] {} (m-4-5)
              edge node [auto] {} (m-5-4);
\end{tikzpicture}
\end{equation}
Next note that $R\to Y$ and $S\to Y$ induce $R\oplus S\to Y$ which is injective with cokernel $Q$. So we have s.e.s.:
$0\to R\oplus S\to Y\to Q\to 0$.
Apply the functor $Hom(-, P)$ to this exact sequence and get l.e.s. assocciated to it :

\begin{equation}
\label{les}
\cdots\to Hom(R\oplus S, P)\xrightarrow{\alpha} Ext^1(Q, P)\xrightarrow{\beta} Ext^1(Y, P)\xrightarrow{\gamma} Ext^1(R\oplus S, P)\xrightarrow{\delta_Y} Ext^2(Q, P).
\end{equation}
We will expand on the morphisms in this exact sequence in Remark \ref{sol}.  
Consider the following exact sequence $0\to P\oplus P\to E\oplus H\to R\oplus S\to 0.$ 
Pushforward this by $\nabla: P\oplus P\to P$ where $\nabla(p_1,p_2)=p_1+p_2$. So we get se.s. $0\to P\to W\to R\oplus S\to 0$ such that following diagram commutes with exact rows: 
\begin{equation}
\label{6}
\begin{tikzpicture}
[back line/.style={densely dotted},cross line/.style={loosely dotted}]
\matrix (m) [matrix of math nodes, row sep=1 em,column sep=2.5em, text height=1.5ex, text depth=0.50ex]
 {
 0 & P\oplus P & E\oplus H  & R\oplus S & 0\\
 0 & P & W & R\oplus S & 0\\
  };

\path[->,font=\scriptsize]
 (m-1-1) edge node [auto] {} (m-1-2)
  (m-1-2) edge node [auto] {$\nabla$} (m-2-2)
   (m-1-2) edge node [auto] {} (m-1-3)
  (m-1-3) edge node [auto] {} (m-1-4)
              edge node [auto] {} (m-2-3)
   (m-1-4) edge node [auto] {} (m-1-5)
           edge node [auto] {id} (m-2-4)
   (m-2-1) edge node [auto] {} (m-2-2)
  (m-2-2) edge node [auto]{} (m-2-3) 
   (m-2-3) edge node [auto] {} (m-2-4)
  (m-2-4) edge node [auto] {} (m-2-5) ;
\end{tikzpicture}
\end{equation}
Now $[W]\in Ext^1(R\oplus S, P)$ and $\delta_Y([W])$ is the image of $[W]$ under the connecting morphism $Ext^1(R\oplus S, P)\xrightarrow{\delta_Y} Ext^2(Q, P)$.
In the above notation we have following intermediate proposition.
\begin{proposition} 
\label{2.1}
 $\delta_Y ([W])=0$ iff we can complete the original diagram \ref{1} to \ref{2}.
\end{proposition}
\begin{remark}
The proof will be given as though the abelian category $\mathcal{A}$ is embedded in the category of $R$-Modules for a commutative ring $R$. But a similar argument can be given for general abelian category using just arrows. 
\end{remark}

We will reinterpret the condition $\delta_Y ([W])=0$ as follows:

Consider
\center $\delta_Y([W])= [0\to P\to W\to R\oplus S\to 0]+[0\to R\oplus S\to Y\to Q\to 0]$ \\
= $[0\to P\to W\to Y\to Q\to 0]$

Now by Grothendieck's criterion in~\cite{SGA 7}, 9.3:  $\delta_Y([W])=0$ iff $\exists$ $X$ such that the following diagram with exact rows and columns commutes:
 \begin{equation}
 \label{7}
\begin{tikzpicture}
[back line/.style={densely dotted},cross line/.style={loosely dotted}]
\matrix (m) [matrix of math nodes, row sep=2.0 em,column sep=2.em, text height=1.5ex, text depth=0.50ex]
 {& 0 & 0 & 0 &\\
 0 & P & W & R\oplus S & 0\\
 0 & P & X & Y & 0\\
  &  & Q & Q & 0\\
  &  & 0 & 0 &\\};
 \path[->, densely dashed]  (m-2-3) edge node [auto] {} (m-3-3)
 (m-3-2) edge node [auto]{} (m-3-3)
  (m-3-3) edge node [auto] {} (m-3-4)
   edge node [auto] {} (m-4-3)   ;
\path[->,font=\scriptsize]
  (m-1-2) edge node[auto] {} (m-2-2)
  (m-1-3) edge node [auto] {} (m-2-3)
(m-1-4) edge node [auto] {} (m-2-4)
 (m-2-1) edge node [auto] {} (m-2-2)
  (m-2-2) edge node [auto] {} (m-2-3)
    edge node [auto] {id} (m-3-2)
  (m-2-3) edge node [auto] {} (m-2-4)
   (m-2-4) edge node [auto] {} (m-2-5)
           edge node [auto] {} (m-3-4)
   (m-3-1) edge node [auto] {} (m-3-2)
  (m-3-4) edge node [auto] {} (m-3-5) 
              edge node [auto] {} (m-4-4)   
  (m-4-3) edge node [auto] {id} (m-4-4)
               edge node [auto] {} (m-5-3)       
  (m-4-4) edge node [auto] {} (m-4-5)
              edge node [auto] {} (m-5-4);
\end{tikzpicture}
\end{equation}

Hence it will be enough to show the following claim.
\paragraph{\textbf{ Claim:}}  The diagram \ref{7} exists iff diagram \ref{2} exists.\\
\begin{proof}[\textbf{Proof of the claim:}]
 Proof of the only if part:\\
Assume we have diagram \ref{7}. Combining diagrams \ref{4}, \ref{6} and \ref{7} we get

\begin{equation}
\label{8}
\begin{tikzpicture}
[back line/.style={densely dotted},cross line/.style={loosely dotted}]
\matrix (m) [matrix of math nodes, row sep=2. em,column sep=2.em, text height=1.5ex, text depth=0.50ex]
 {  & 0 &    0            &     &    &     &0 & \\
 0 & P & E              &     &    &     & R & 0\\
 0 & H & E\oplus H&      &    &    &    &  \\
      &     &              &W  &    &     &    &\\
     &    &                &     & X &     &    &\\
     &    &                &     &    & Y  & F & 0\\
 0 & S &                 &     &    & G & Q & 0\\
    & 0 &                 &     &    & 0   & 0 &\\};
\path[->, densely dashed]  (m-2-3) edge [bend left=30] node [auto] {} (m-5-5)
 (m-5-5) edge node [auto]{} (m-6-6)
 (m-3-2) edge[bend right=30] node [auto]{} (m-5-5);
 
\path[->,font=\scriptsize]
  (m-1-2) edge node[auto] {} (m-2-2)
  (m-1-3) edge node [auto] {} (m-2-3)
(m-1-7) edge node [auto] {} (m-2-7)
 (m-2-1) edge node [auto] {} (m-2-2)
  (m-2-2) edge node [auto] {} (m-2-3)
    edge node [auto] {} (m-3-2)
  (m-2-3) edge node [auto] {} (m-2-7)
             edge[auto] node [auto] {} (m-3-3)
   (m-2-7) edge node [auto] {} (m-6-6)
           edge node [auto] {} (m-6-7)
           (m-2-7) edge node [auto] {} (m-2-8)
   (m-3-1) edge node [auto] {} (m-3-2)
  (m-3-2) edge node [auto]{} (m-3-3) 
    edge node [auto]{} (m-7-2) 
   (m-3-3) edge node [auto] {} (m-4-4)
  (m-4-4) edge node [auto] {} (m-5-5) 
               (m-6-6) edge node [auto] {} (m-6-7) 
              edge node [auto] {} (m-7-6) 
               (m-6-7) edge node [auto] {} (m-6-8) 
              edge node [auto] {} (m-7-7)   
  (m-7-1) edge node [auto] {} (m-7-2) 
     (m-7-2) edge node [auto]{}(m-7-6)
               edge node [auto]{}(m-8-2)
               edge node [auto]{}(m-6-6)
  (m-7-7) edge node [auto] {} (m-7-8)
               edge node [auto] {} (m-8-7)       
  (m-7-6) edge node [auto] {} (m-7-7)
              edge node [auto] {} (m-8-6);
\end{tikzpicture}
\end{equation}

which when squashed looks like the diagram \ref{2}. 
\begin{equation}
\label{9}
\begin{tikzpicture}
[back line/.style={densely dotted},cross line/.style={loosely dotted}]
\matrix (m) [matrix of math nodes, row sep=2. em,column sep=2.em, text height=1.5ex, text depth=0.50ex]
 {& 0 & 0 & 0 &\\
 0 & P & E & R & 0\\
 0 & H & X & F & 0\\
 0 & S & G & Q & 0\\
  & 0 & 0 & 0 &\\};
 \path[->, densely dashed]  (m-2-3) edge node [auto] {} (m-3-3)
 (m-3-2) edge node [auto]{} (m-3-3)
  (m-3-3) edge node [auto] {} (m-3-4)
   edge node [auto] {} (m-4-3)   ;
\path[->,font=\scriptsize]
  (m-1-2) edge node[auto] {} (m-2-2)
  (m-1-3) edge node [auto] {} (m-2-3)
(m-1-4) edge node [auto] {} (m-2-4)
 (m-2-1) edge node [auto] {} (m-2-2)
  (m-2-2) edge node [auto] {} (m-2-3)
    edge node [auto] {} (m-3-2)
  (m-2-3) edge node [auto] {} (m-2-4)
   (m-2-4) edge node [auto] {} (m-2-5)
           edge node [auto] {} (m-3-4)
     
   (m-3-1) edge node [auto] {} (m-3-2)
  (m-3-2) edge node [auto]{} (m-4-2) 
  (m-3-4) edge node [auto] {} (m-3-5) 
              edge node [auto] {} (m-4-4)   
  (m-4-1) edge node [auto] {} (m-4-2) 
     (m-4-2) edge node [auto]{}(m-4-3)
               edge node [auto]{}(m-5-2)
  (m-4-3) edge node [auto] {} (m-4-4)
               edge node [auto] {} (m-5-3)       
  (m-4-4) edge node [auto] {} (m-4-5)
              edge node [auto] {} (m-5-4);
\end{tikzpicture}
\end{equation}

 Now we show that the rows and columns in this newly formed diagram \ref{9} are exact. 
 \begin{enumerate}
 \item
Consider the sequence $0\to E\to X\to G\to 0:$
\\
\textbf{Exactness at $E$:}
For $e\in E$ such that $ e\mapsto 0\in X$ then via $X\to Y$ it further goes to $0\in Y$. Let $e\mapsto r\in R.$ Then $r\mapsto 0$ via $R\to Y$. But $R \to Y$ is injective. So $e\mapsto 0$ in $R$. Thus $e$ is in the image of $P\to E$, say image of $p\in P$. Then $p\mapsto 0$ by $P\to E\to X$. But $P\to X$ is injective. Hence $p=0 \in P$. Hence $e=0\in E$. \\
\textbf{Exactness at $G$:}
Further $X\to G$ is surjective as it factors as $X\to Y\to G$ each of which is surjective. \\
\textbf{Exactness at $X$:}
First the composition $E\to X\to G$ is 0 as it factors through $R\to Y\to G$ which is 0. 
Now let $x\in X$ such that $x\mapsto 0$ via $X\to G$. Let $y$ be the image of $x$ via $X\to Y.$ Now $y\mapsto 0$ by $Y\to G$. Therefore there exists $r\in R$ such that $r\mapsto y$ via $R\to Y$. Let $e\in E $ be such that $e\mapsto r$ via $E\to R.$ Let $x'$ be the image of $e$ in $X$. So $x-x'\in X$ such that $x-x'\mapsto 0$ via $X\to Y$. Therefore there exists $p\in P$ such that $p\mapsto x-x'$ via $P\to X.$ Therefore for $Im(p)\in E, Im(p)+e\mapsto x$ via $E\to X$. 
\item
Similarly using other part of diagram \ref{8} one can show that $0\to H\to X\to F\to0$ is exact. 
\end{enumerate}

Next the proof for the if part, assume that the diagram \ref{2} exists. We need to show that $X$ in diagram \ref{2} fits into the diagram \ref{7}.
First we see that the respective maps in diagram \ref{7} exist. Since 

\begin{equation}
\label{10}
\begin{tikzpicture}
[back line/.style={densely dotted},cross line/.style={loosely dotted}]
\matrix (m) [matrix of math nodes, row sep=2 em,column sep=2.em, text height=1.5ex, text depth=0.50ex]
 {X &     &\\
      & Y & G \\
      & F & Q\\
  };
\path[->]  (m-1-1) edge [bend left=30] node [auto] {} (m-2-3)
 (m-1-1) edge[densely dashed] node [auto]{} (m-2-2)
 (m-1-1) edge[bend right=30] node [auto]{} (m-3-2);
\path[->,font=\scriptsize]
  (m-2-2) edge node [auto]{} (m-2-3) 
   edge node [auto]{} (m-3-2) 
   (m-2-3) edge node [auto] {} (m-3-3)
     (m-3-2) edge node [auto]{}(m-3-3);
\end{tikzpicture}
\end{equation}

such that inner square is a pullback digram and the outer square is commutative, hence $\exists$ a map $X\to Y$ such that the diagram \ref{10} commutes. 
Similarly $W$ is a pushforward of a certain diagram and hence $\exists$ a map $W\to X$. So we have all the arrows in diagram \ref{7}. Next to show that all the rows and columns are exact. Enough to show that $0\to P\to X\to Y\to 0$ and $0\to W\to X\to Q\to 0$ are exact.
\begin{enumerate}
\item Consider the sequence $0\to P\to X\to Y\to 0$ : \\
\textbf{Exacteness at P:} It is exact at $P$ as $P\to X$ factors through $P\to E$ and $E\to X$ each of them is injective. \\
\textbf{Exactness at $Y$:} Let $y\in Y$. say $y\mapsto g\in G$. Since $X\to G$ is surjective, $\exists$ $x\in X$ such that $x\mapsto g$. Now Let $x\mapsto y'\in Y$ Then $y-y'\mapsto 0$ via Y$\to G$. Therefore $\exists$ $ e\in E$ such that $e\mapsto r\mapsto y-y'$ via $E\to R\to Y$. Let $Im(e)\in X$. then $x-Im(e)\in X$ such that $x-Im(e)\mapsto y'+y-y'=y\in Y$. Thus $X\to Y$ is surjective.\\
\textbf{Exactness at $X$:} First the composition is 0 as $P\to X\to Y$ factors as $P\to E\to R\to Y$ which is 0. For the other part of the exactness, let $x\in X$ such that $x\mapsto 0$ via $X\to Y.$ Then $x\mapsto 0$ via $X\to G$, hence $\exists$ $e\in E$ such that $e\mapsto x$. Now $e\mapsto 0$ via $E\to R$ as $e\mapsto x\mapsto 0$ via $E\to X\to F$ and $R\to F$ is injective. Therefore $\exists$ $p\in P$ such that $p\mapsto e$. Therefore $p\mapsto x.$ 
\item Now the exactness of the second sequence $0\to W\to X\to Q\to 0$ follows from similar diagram chase.
\end{enumerate}¥ 
\end{proof}

Next we relate the element $\delta_Y([W])$ to the Baer sum in Theorem \ref{1.1}(ii). Assuming the diagram \ref{1} we can form the following diagram :
\begin{equation}
\label{11}
\begin{tikzpicture}
[back line/.style={densely dotted},cross line/.style={loosely dotted}]
\matrix (m) [matrix of math nodes, row sep=3 em,column sep=5.em, text height=1.5ex, text depth=0.50ex]
 {
 Ext^1(R, P)\oplus  Ext^1(S, P)\oplus Ext^1(Q, R)\oplus Ext^1(Q, S) & Ext^2(Q, P)\oplus Ext^2(Q, P)\\
 Ext^1(R\oplus S, P)\oplus Ext^1(Q, R\oplus S) & Ext^2(Q, P)\\
  };

\path[->,font=\scriptsize]
 
 (m-1-1) edge node [auto] {$\phi$} (m-1-2)
  (m-1-1) edge node [auto] {($\nabla_{\ast},\Delta^{\ast}$)} (m-2-1)
   (m-1-2) edge node [auto] {$+$} (m-2-2)
  (m-2-1) edge node [auto] {$\psi$} (m-2-2);
\end{tikzpicture}
\end{equation}

where the maps are defined as follows:
\begin{definition}
Let $[E] \in Ext^1(R, P), [H] \in Ext^1(S, P), [F] \in Ext^1(Q, R), [G] \in Ext^1(Q, S)$. Then
  \begin{center}
  $\phi( [E], [H], [F], [G]):= ([E]\cup [F], [H]\cup [G]),$
  $(\nabla_{\ast},\Delta^{\ast})([E], [H], [F], [G]):= (\nabla_{\ast}([E]\oplus[H]), \Delta^{\ast}([F]\oplus[G]))$ 
 \end{center}
 where $[E]\oplus [H]:= [0\to P\oplus P\to E\oplus H\to R\oplus S\to 0]$
   and\\
    $\nabla_{\ast}([E]\oplus[H])$ is the class of pushforward of the exact sequence 
   \begin{center}
   $0\to P\oplus P\to E\oplus H\to R\oplus S\to 0$ 
   \end{center}
   by the morphism $\nabla: P\oplus P \to P$ given by addition. 
   
$\Delta^{\ast}([F]\oplus[G])$= the class of the pullback of the exact sequence $0\to R\oplus S\to F\oplus G\to Q\oplus Q\to 0$ by the morphism $\Delta : Q\to Q\oplus Q$ given by the diagonal.

If [$W]\in Ext^1(R\oplus S, P)$ and $[Y]\in Ext^1(Q, R\oplus S)$, then 
\begin{center}
$\psi([W], [Y])$:= the class of the exact sequence 
$0\to P\to W\to Y\to Q \to 0.$
\end{center}

Next, for $[0\to P\to T_1\to Z_1\to Q\to 0]$ and $[0\to P\to T_2\to Z_2\to Q\to 0] \in Ext^2(Q, P)$,

$[0\to P\to T_1\to Z_1\to Q\to 0] + [0\to P\to T_2\to Z_2\to Q\to 0]$ is given by their Baer sum as follows:

Let $T$ be the pushforward of $P\to T_1$ and $P\to T_2 , Z $ is the pullback of $Z_1\to Q$ and $Z_2\to Q.$ Then the Baer sum is the class of the exact sequence: $[0\to P\to T\to Z\to Q\to 0]$.
\end{definition}
\begin{lemma}
\label{2.4}
The diagram \ref{11} above is commutative.
\end{lemma}
\begin{proof}
Let $[E] \in Ext^1(R, P), [H] \in Ext^1(S, P), [F] \in Ext^1(Q, R), [G] \in Ext^1(Q, S)$. Following the top arrow gives us that the image of $([E], [H], [F], [G])$ by the composition is the Baer sum of 
\begin{center}
$o\to P\to E \to F\to Q\to 0$ and $0\to P\to H \to G \to Q \to 0$
\end{center} 
which is given by the following exact sequence $0\to P\to T\to Y\to Q\to 0$
where T is the pushforward of $P\to E$ and $P\to H , Y $ is the pullback of $G\to Q$ and $F\to Q.$
So we have the following diagrams:
\begin{equation}
\label{12}
\begin{tikzpicture}
[back line/.style={densely dotted},cross line/.style={loosely dotted}]
\matrix (m) [matrix of math nodes, row sep=2 em,column sep=2.em, text height=1.5ex, text depth=0.50ex]
 {
 0 & P & E & R & 0\\
 0 & H & T &  R & 0\\
  };

\path[->,font=\scriptsize]
 (m-1-1) edge node [auto] {} (m-1-2)
  (m-1-2) edge node [auto] {} (m-2-2)
   (m-1-2) edge node [auto] {} (m-1-3)
  (m-1-3) edge node [auto] {} (m-1-4)
              edge node [auto] {} (m-2-3)
   (m-1-4) edge node [auto] {} (m-1-5)
           edge node [auto] {id} (m-2-4)
   (m-2-1) edge node [auto] {} (m-2-2)
  (m-2-2) edge node [auto]{} (m-2-3) 
   (m-2-3) edge node [auto] {} (m-2-4)
  (m-2-4) edge node [auto] {} (m-2-5) ;
\end{tikzpicture}
\end{equation}
\begin{equation}
\label{13}
\begin{tikzpicture}
[back line/.style={densely dotted},cross line/.style={loosely dotted}]
\matrix (m) [matrix of math nodes, row sep=2 em,column sep=2.em, text height=1.5ex, text depth=0.50ex]
 {
 0 & S & Y & G & 0\\
 0 & S & F & Q & 0\\
  };

\path[->,font=\scriptsize]
 (m-1-1) edge node [auto] {} (m-1-2)
  (m-1-2) edge node [auto] {id} (m-2-2)
   (m-1-2) edge node [auto] {} (m-1-3)
  (m-1-3) edge node [auto] {} (m-1-4)
              edge node [auto] {} (m-2-3)
   (m-1-4) edge node [auto] {} (m-1-5)
           edge node [auto] {} (m-2-4)
   (m-2-1) edge node [auto] {} (m-2-2)
  (m-2-2) edge node [auto]{} (m-2-3) 
   (m-2-3) edge node [auto] {} (m-2-4)
  (m-2-4) edge node [auto] {} (m-2-5) ;
\end{tikzpicture}
\end{equation}

Following the left arrow in diagram \ref{11}, we get $\nabla_{\ast}([E]\oplus[H])$ which fits into the diagram

\begin{equation}
\label{14}
\begin{tikzpicture}
[back line/.style={densely dotted},cross line/.style={loosely dotted}]
\matrix (m) [matrix of math nodes, row sep=2 em,column sep=2.em, text height=1.5ex, text depth=0.50ex]
 {
 0 & P\oplus P & E\oplus H & R\oplus S & 0\\
 0 & P & W & R\oplus S & 0\\
  };

\path[->,font=\scriptsize]
 (m-1-1) edge node [auto] {} (m-1-2)
  (m-1-2) edge node [auto] {$\nabla$} (m-2-2)
   (m-1-2) edge node [auto] {} (m-1-3)
  (m-1-3) edge node [auto] {} (m-1-4)
              edge node [auto] {} (m-2-3)
   (m-1-4) edge node [auto] {} (m-1-5)
           edge node [auto] {id} (m-2-4)
   (m-2-1) edge node [auto] {} (m-2-2)
  (m-2-2) edge node [auto]{} (m-2-3) 
   (m-2-3) edge node [auto] {} (m-2-4)
  (m-2-4) edge node [auto] {} (m-2-5) ;
\end{tikzpicture}
\end{equation}
and $\Delta^{\ast}([F]\oplus[G])$ fits into 
\begin{equation}
\label{15}
\begin{tikzpicture}
[back line/.style={densely dotted},cross line/.style={loosely dotted}]
\matrix (m) [matrix of math nodes, row sep=2 em,column sep=2.em, text height=1.5ex, text depth=0.50ex]
 {
 0 & R\oplus S & Z & Q & 0\\
 0 & R\oplus S & F\oplus G & Q\oplus Q & 0\\
  };

\path[->,font=\scriptsize]
 (m-1-1) edge node [auto] {} (m-1-2)
  (m-1-2) edge node [auto] {id} (m-2-2)
   (m-1-2) edge node [auto] {} (m-1-3)
  (m-1-3) edge node [auto] {} (m-1-4)
              edge node [auto] {} (m-2-3)
   (m-1-4) edge node [auto] {} (m-1-5)
           edge node [auto] {$\Delta$} (m-2-4)
   (m-2-1) edge node [auto] {} (m-2-2)
  (m-2-2) edge node [auto]{} (m-2-3) 
   (m-2-3) edge node [auto] {} (m-2-4)
  (m-2-4) edge node [auto] {} (m-2-5) ;
\end{tikzpicture}
\end{equation}
So 
$\psi\circ(\nabla_{\ast},\Delta^{\ast})([E], [H], [F], [G])$ = the class of the exact sequence
$0\to P\to W \to Z\to Q\to 0.$
 Thus we have two elements in $Ext^2(Q,P)$: 
 \begin{center}
 $0\to P\to W \to Z\to Q\to 0$
 and\\ $0\to P\to T \to Y\to Q\to 0$
 \end{center}

 We need to show that these two elements are same in $Ext^2(Q, P)$.
 To see this enough to give a commutative diagram as follows:
 \begin{equation}
 \label{16}
\begin{tikzpicture}
[back line/.style={densely dotted},cross line/.style={loosely dotted}]
\matrix (m) [matrix of math nodes, row sep=2 em,column sep=2.em, text height=1.5ex, text depth=0.50ex]
 {
 0 & P & W & Z & Q & 0\\
 0 & P & T & Y & Q & 0\\
  };
 \path[->, densely dashed] (m-1-3) edge node [auto] {} (m-2-3)
                               (m-1-4) edge node [auto] {} (m-2-4);
\path[->,font=\scriptsize]
 (m-1-1) edge node [auto] {} (m-1-2)
  (m-1-2) edge node [auto] {id} (m-2-2)
   (m-1-2) edge node [auto] {} (m-1-3)
  (m-1-3) edge node [auto] {} (m-1-4)
            
   (m-1-4) edge node [auto] {} (m-1-5)
           
           (m-1-5) edge node [auto] {} (m-1-6)
           edge node [auto] {id} (m-2-5)
   (m-2-1) edge node [auto] {} (m-2-2)
  (m-2-2) edge node [auto]{} (m-2-3) 
   (m-2-3) edge node [auto] {} (m-2-4)
  (m-2-4) edge node [auto] {} (m-2-5)
  (m-2-5) edge node [auto] {} (m-2-6) ;
\end{tikzpicture}
\end{equation}
Using the diagram \ref{12} and \ref{14} for the universal properties of $T$ and $W$ we can show that there is a map $W\to T$ such that the diagram commutes. And similarly using the diagrams \ref{13} and \ref{15} for $Y$ and $Z$ can show we can complete the diagram \ref{16}. Hence the proposition follows.
\end{proof}

\begin{proposition}
\label{2.5}
Recall $\delta_Y([W])$ from Proposition \ref{2.1}. Then $\delta_Y([W])$= Baer sum of $[E]\cup [F]$ and $[H]\cup [G]$.
\end{proposition}
\begin{proof} From the proof of the Lemma \ref{2.4}. 
$\psi\circ(\nabla_{\ast},\Delta^{\ast})([E], [H], [F], [G])=$ the class of the exact sequence $0\to P\to W \to Z\to Q\to 0,$ where recall that $Z$ is the pullback of $F\oplus G\to Q\oplus Q$ and $Q\xrightarrow{\Delta} Q\oplus Q.$ But $Z\to F\oplus G$ is given by the pair $Z\to F$ and $Z\to G$ such that $Z\to F\to Q$ is same as $Z\to Q$ which is same as $Z\to G\to Q.$ Hence by universal property of pullback for $Y$ as in diagram \ref{3}, there is a morphism $Z\to Y.$ Now $\delta_Y([W])$ is given by the class of $0\to R\oplus S\to Y\to Q\to 0$. Now we can compare the two exact sequences by the following diagram:
 \begin{equation}
 \label{17}
\begin{tikzpicture}
[back line/.style={densely dotted},cross line/.style={loosely dotted}]
\matrix (m) [matrix of math nodes, row sep=2 em,column sep=2.em, text height=1.5ex, text depth=0.50ex]
 {
  0 & R\oplus S & W & Z & Q & 0\\
  0 & R\oplus S & W & Y & Q & 0\\
  };
 \path[->, densely dashed] 
                               ;
\path[->,font=\scriptsize]
 (m-1-1) edge node [auto] {} (m-1-2)
  (m-1-2) edge node [auto] {id} (m-2-2)
   (m-1-2) edge node [auto] {} (m-1-3)
  (m-1-3) edge node [auto] {} (m-1-4)
            (m-1-3) edge node [auto] {} (m-2-3)
   (m-1-4) edge node [auto] {} (m-1-5)
           (m-1-4) edge node [auto] {} (m-2-4)
           (m-1-5) edge node [auto] {} (m-1-6)
           edge node [auto] {id} (m-2-5)
   (m-2-1) edge node [auto] {} (m-2-2)
  (m-2-2) edge node [auto]{} (m-2-3) 
   (m-2-3) edge node [auto] {} (m-2-4)
  (m-2-4) edge node [auto] {} (m-2-5)
  (m-2-5) edge node [auto] {} (m-2-6) ;
\end{tikzpicture}
\end{equation}

  Now Lemma \ref{2.4} implies that $\psi\circ(\nabla_{\ast},\Delta^{\ast})([E], [H], [F], [G])$ is the Baer sum of $[E]\cup [F]$ and $[H]\cup [G]$. Hence the proposition follows.
\end{proof}

\begin{proof}[\textbf{Proof of Theorem \ref{1.1}}]
The equivalence of statements follows from the Propositions \ref{2.1} and \ref{2.5}.
\end{proof}
\begin{remark}[Isomorphism classes of the solution set]
\label{sol}
Note that $\delta_Y([W])=0$ iff there exists $[X]\in Ext^1(Y, P)$ such that $\gamma([X])=[W]$, by the exact sequence~\ref{les}.
Any other $[X']\in Ext^1(Y, P)$ such that $\gamma([X'])=[W]$ would imply that $[X]-[X']\in$ Ker $\gamma.$ Thus the solution set of $[X] \in Ext^1(Y, P)$ such that $\gamma([X])=[W]$ is given by $[X]+$Ker $\gamma$, where $[X]$ is one solution. Further Ker $\gamma=$ Im $\beta$ which is isomorphic to Coker $\alpha.$ 

Let us recall the map $\alpha$. Given $\phi:R\oplus S\to P$, consider the pushforward of the exact sequence $0\to R\oplus S\to Y\to Q\to 0$ by $\phi$, then $\alpha(\phi)= \phi_{\ast}([0\to R\oplus S\to Y\to Q\to 0])$. We observed from diagrams \ref{15} and \ref{17} that $[0\to R\oplus S\to Y\to Q\to 0]=  \Delta^{\ast}([F]\oplus[G])$. Hence 
\begin{center}
$\alpha: Hom(R\oplus S, P)\to Ext^1(Q, P)$ 
\end{center}
is given by $\alpha(\phi)= \phi_{\ast}(\Delta^{\ast}([F]\oplus[G]))$. On the other hand, consider the exact sequence $0\to R\to F\to Q\to 0$. Apply $Hom(-, P)$ we get 
\begin{center}
$\delta_F: Hom (R, P)\to Ext^1(Q, P)$ and similarly $\delta_G: Hom (S, P)\to Ext^1(Q, P)$ 
\end{center}
two connecting homomorphisms, Then 
\begin{center}$\alpha(f+g)=(f+g)_{\ast}(\Delta^{\ast}([F]\oplus[G]))$ = Baer sum of  $\delta_F(f)$ and $\delta_G(g).$ 
\end{center}
From above discussion, we can observe that the Set of $[X]\in Ext^1(Y, P)$ such that $X$ fits into the diagram \ref{2} is a principal homogeneous space under the abelian group $\dfrac{Ext^1(Q, P)}{Im (\delta_F+\delta_G)}$, where it acts on the solution set as follows:
Given $\overline{\lambda}\in \dfrac{Ext^1(Q, P)}{Im (\delta_F+\delta_G)}$ and $X$ in diagram \ref{2}, hence $[X]\in Ext^1(Y, P)$,
 $\overline{\lambda}[X]:= \beta(\lambda)+[X].$\\
 In particular if $\alpha$ is surjective, then there is a unique solution $X$ for diagram \ref{1} up to isomorphism, if it exists. 
\end{remark}

\end{section}


\bibliographystyle{}

\medskip
\medskip

Rakesh Pawar, \textsc{School of Mathematics,
  Tata Institute of Fundamental Research, Mumbai,
  Mumbai - 400 005, India.
  }\par\nopagebreak
  \textit{E-mail address}: \texttt{Email: rpawar@math.tifr.res.in}

\end{document}